\documentclass[12pt]{amsart}
\usepackage{amssymb,amsfonts,latexsym,amscd}

\usepackage[all,cmtip,2cell]{xy}
\UseAllTwocells
\usepackage{verbatim}

\newtheorem{theorem}{Theorem}[section]
\newtheorem{lemma}[theorem]{Lemma}

\newtheorem{corollary}[theorem]{Corollary}
\theoremstyle{definition}
\newtheorem{definition}[theorem]{Definition}

\newtheorem{example}[theorem]{Example}
\newtheorem{question}[theorem]{Question}

\newtheorem{remark}[theorem]{Remark}

\newcommand{\id}{\text{id}}

\newcommand{\Ker}{\text{Ker\,}}

\newcommand{\End}{{\rm End}}
\newcommand{\Fun}{\text{Fun}}
\newcommand{\Irr}{\text{\rm Irr}}

\newcommand{\FPdim}{\text{\rm FPdim}}

\newcommand{\Hom}{{\rm Hom}}

\newcommand{\Rep}{\text{Rep}}

\newcommand{\Vect}{{\rm Vec}}

\newcommand{\A}{\mathcal{A}}
\newcommand{\B}{\mathcal{B}}
\newcommand{\C}{\mathcal{C}}
\newcommand{\D}{\mathcal{D}}
\newcommand{\M}{\mathcal{M}}
\newcommand{\N}{\mathcal{N}}

\newcommand{\ot}{\otimes}

\newcommand{\ben}{\begin{enumerate}}
\newcommand{\een}{\end{enumerate}}

\newcommand{\Z}{{\mathcal Z}}

\newcommand{\Mod}{\mbox{Mod}}
\newcommand{\Bimod}{\mbox{Bimod}}

\begin{document}

\title[Exact factorizations and extensions of fusion categories] {Exact factorizations and extensions of fusion categories}

\author{Shlomo Gelaki}
\address{Department of Mathematics, Technion-Israel Institute of
Technology, Haifa 32000, Israel} \email{gelaki@math.technion.ac.il}

\date{\today}


\begin{abstract}
We introduce and study the new notion of an {\em exact factorization} $\B=\A\bullet \C$ of a fusion category $\B$ into a product of two fusion subcategories $\A,\C\subseteq \B$ of $\B$. This is a categorical generalization of the well known notion of an exact factorization of a finite group into a product of two subgroups. We then relate exact factorizations of fusion categories to exact sequences of fusion categories with respect to an indecomposable module category, which was introduced and studied in \cite{EG}. We also apply our results to study extensions of a group-theoretical fusion category by another one, provide some examples, and propose a few natural questions. 
\end{abstract}

\maketitle

\section{introduction}

Finite groups with exact factorization are fundamental objects in group theory which naturally show up in many interesting results in the theories of Hopf algebras and tensor categories (see, e.g., \cite{BGM}, \cite{K}, \cite{M}, \cite{Na}, \cite{O}). 
Recall that a finite group $G$ admits an exact factorization $G=G_1G_2$ into a product of two subgroups $G_1,G_2\subseteq G$ of $G$ if $G_1$ and $G_2$ intersect trivially and the order of $G$ is the product of the orders of $G_1$ and $G_2$. Equivalently, $G=G_1G_2$ is an exact factorization if every element $g\in G$ can be uniquely written in the form $g=g_1g_2$, where $g_1\in G_1$ and $g_2\in G_2$.

Our first goal in this paper is to provide a categorical generalization of the notion of exact factorizations of finite groups. More precisely, we introduce and study the new notion of an {\em exact factorization} $\B=\A\bullet \C$ of a fusion category $\B$ into a product of two fusion subcategories $\A,\C\subseteq \B$ of $\B$. We say that $\B=\A\bullet \C$ is an exact factorization if $\A\cap \C=\Vect$ and $\FPdim(\B)= \FPdim(\A)\FPdim(\C)$. Then in Theorem \ref{exfac} we prove that $\B=\A\bullet \C$ if and only if every simple object of $\B$ can be uniquely expressed in the form $X\ot Y$, where $X,Y$ are simple objects of $\A,\C$, respectively. For example, exact factorizations $\B=\Vect(G_1)\bullet \Vect(G_2)$ 
are classified by groups $G$ with exact 
factorization $G=G_1G_2$ and a cohomology class $\omega\in H^3(G,k^{\times})$ 
which is trivial on $G_1$ and $G_2$ (but not necessarily on $G$). 

Recall next that the theory of exact sequences of tensor categories 
was introduced by A. Brugui\`eres and S. Natale
\cite{BN, BN1} as a categorical generalization of the theory of exact sequences of Hopf algebras.
In their definition of an exact sequence of tensor categories
\begin{equation*}\label{exactseq}
\mathcal{A}\xrightarrow{\iota}\mathcal{B}\xrightarrow{F}\mathcal{C}
\end{equation*} 
the category $\mathcal{A}$ 
is forced to have a tensor functor to ${\rm Vec}$ (so to be the representation category of a Hopf algebra). Later on in \cite{EG} we generalized the definition of \cite{BN} further to 
eliminate this drawback, and in particular to include the example of the Deligne tensor product
$\mathcal{B}:=\mathcal{A}\boxtimes \mathcal{C}$
for any finite tensor categories $\mathcal{A}$, $\mathcal{C}$. 
We did so by replacing the category ${\rm Vec}$ 
by the category ${\rm End}(\mathcal{N})$ of endofunctors of 
an indecomposable $\mathcal{A}$-module category $\mathcal{N}$, 
and defined the notion of an {\em exact sequence 
\begin{equation}\label{exseqfc}
\mathcal{A}\xrightarrow{\iota} \mathcal{B}\xrightarrow{F} \mathcal{C}\boxtimes \End(\N)
\end{equation}
with respect to $\mathcal{N}$}. We showed that the dual of an exact sequence is again an exact sequence. We also showed that for any exact sequence (\ref{exseqfc}), 
${\rm FPdim}(\mathcal{B})={\rm FPdim}(\mathcal{A}){\rm FPdim}(\mathcal{C})$, and that this property in fact 
characterizes exact sequences (provided that $\iota$ is injective, $F$ is surjective, and $\A\subseteq {\rm Ker}(F)$). 
Moreover, we showed that if in an exact sequence (\ref{exseqfc}), $\A$ and $\C$ are fusion categories, then so is $\B$.

Our second goal in this paper is to relate exact factorizations of fusion categories with exact sequences.
More precisely, in Theorem \ref{mainnew} we prove that an exact sequence of fusion categories (\ref{exseqfc}) defines  
an exact factorization $\mathcal{B}^*_{\C\boxtimes \mathcal{N}}=\C\bullet \A^*_{\N}$, and vice versa, 
any exact factorization $\B=\A\bullet \C$ of fusion categories gives rise to an exact sequence (\ref{exseqfc}) with respect to any indecomposable module category $\mathcal{N}$ over $\mathcal{A}$.

The structure of the paper is as follows. In Section 2 we recall some necessary background on module categories over fusion categories, exact sequences of fusion categories, and the class of group-theoretical fusion categories. In Section 3 we introduce and study factorizations $\B=\A\C$ and exact factorizations $\B=\A\bullet \C$ of a fusion category $\B$ into a product of fusion subcategories $\A,\C\subseteq \B$ of $\B$. In Section 4 we prove  Theorem \ref{mainnew}, in which we relate exact factorizations of fusion categories to exact sequences of fusion categories with respect to an indecomposable module. We then apply our results to group-theoretical fusion categories, and deduce in Corollary \ref{gtext} that any extension of a group-theoretical fusion category by another one is {\em Morita equivalent} to a fusion category with an exact factorization into a product of two group-theoretical fusion subcategories. We also discuss some examples of exact factorizations and exact sequences (e.g., of Kac (quasi-)Hopf algebras in Corollary \ref{kacalg}), and propose some natural questions.

\begin{remark}
We plan to extend the results of this paper to nonsemisimple finite tensor categories in a subsequent paper.
\end{remark}

{\bf Acknowledgements.} I am grateful to Pavel Etingof for stimulating and helpful discussions. 
Part of this work was done while I was visiting the Department of Mathematics at the University of Michigan in Ann Arbor; I am grateful for their warm hospitality.
This work was partially supported by the Israel Science Foundation (grant no. 561/12). 

\section{Preliminaries}
Throughout the paper, $k$ will denote an algebraically closed field of characteristic $0$. All the categories mentioned in this paper are assumed to be $k$-linear abelian and {\em finite}.

We refer the reader to the book \cite{EGNO} as a general reference for the theory of fusion categories.

\subsection{Module categories}
Let $\mathcal{A}$ be a fusion category over $k$ (see \cite{EGNO}).
Let $\M$ be a left semisimple $\A$-module category, and let $\N$ be a right semisimple $\A$-module category. 
Consider the tensor product 
$\mathcal{N} \boxtimes_{\mathcal{A}}\mathcal{M}$ (see \cite{ENO2}).
Namely, if $A_1,A_2$ are algebras in $\A$ such that $\M={\rm mod}-A_1$ and 
$\N=A_2-{\rm mod}$, then $\mathcal{N} \boxtimes_{\mathcal{A}}\mathcal{M}$
is the category of $(A_2,A_1)$-bimodules in $\A$, 
which can also be described as the category 
of left $A_2$-modules in $\M$, or the category 
of right $A_1$-modules in $\N$ (see \cite[Section 7.8]{EGNO}).

Recall that a left semisimple $\A$-module category $\M$ is said to be {\em indecomposable} if it is not a direct sum of
two nonzero module categories. Let $\mathcal{M}$ be an indecomposable 
$\mathcal{A}$-module category\footnote{Here and below, by ``a module category'' we will mean a ``left semisimple module category", unless otherwise specified.}. 
Let $\End(\M)$ be the abelian category of endofunctors of
$\mathcal{M}$, and let $\mathcal{A}_{\mathcal{M}}^*:=
\End_{\mathcal{A}}(\mathcal{M})$ be the dual category of $\A$ with respect to $\M$, i.e., the category of
$\mathcal{A}-$linear endofunctors of $\mathcal{M}$. Recall that composition of functors turns $\End(\M)$ into a multifusion category, 
and $\mathcal{A}_{\mathcal{M}}^*$ into a fusion category.

Let ${\rm Gr}(\A)$ be the Grothendieck ring of $\A$.  
Recall that we have a character $\FPdim: {\rm Gr}(\A)\to \Bbb R$, attaching to $X\in \A$ the Frobenius-Perron dimension $\FPdim(X)$ of $X$. 
Recall also that we have a virtual object $R_\A\in {\rm Gr}(\A)\otimes_{\Bbb Z}\Bbb R$, such that 
$XR_\A=R_\A X=\FPdim(X) R_\A$ for all $X\in {\rm Gr}(\A)$. Namely, we have 
$R_\A=\sum_i \FPdim(X_i)X_i$, where $X_i$ are the simple objects of $\A$. We set $\FPdim(\A):=\FPdim(R_\A)$ \cite{ENO1}.

Also, if $\M$ is an indecomposable $\A$-module category, let 
${\rm Gr}(\M)$ be the Grothendieck group of $\M$. 
Let $\lbrace{ M_j\rbrace}$ be the basis of simple objects of $\M$. It follows from the Frobenius-Perron theorem that there is a unique 
up to scaling element $R_\M\in {\rm Gr}(\M)\otimes_{\Bbb Z}\Bbb R$ such that for every $X\in {\rm Gr}(\A)$, 
$XR_\M=\FPdim(X)R_\M$. Namely, we have that
$$R_\M=\sum_j \FPdim(M_j)M_j.$$ 
The numbers $\FPdim(M_j)$ are defined uniquely up to scaling by the 
property $$\FPdim(X\otimes M)=\FPdim(X)\FPdim(M),\,\,\,X\in \A,\, M\in \M,$$ and it is convenient to normalize them in such a way that $$\FPdim(R_\M)=\FPdim(\A),$$ which we
will do from now on (see \cite[Proposition 3.4.4, Exercise 7.16.8]{EGNO}, and \cite[Subsection 2.2]{ENO2}). 
It is clear that $$R_\A M_j=\FPdim(M_j)R_\M.$$

\subsection{Exact sequences of fusion categories with respect to a module category} 
Let $\A\subseteq \B$, $\C$ be fusion categories, and let $\M$ be an indecomposable module category over $\A$.
Let $F:\B\to \C\boxtimes \End(\M)$ be a surjective (= dominant) functor such that 
$\mathcal{A}=\Ker(F)$ (= the subcategory of $X\in \B$ such that $F(X)\in \End(\M)$).
Recall \cite{EG} that $F$ {\em defines an exact sequence 
\begin{equation}\label{ES}
\mathcal{A}\xrightarrow{\iota} \mathcal{B}\xrightarrow{F} \mathcal{C}\boxtimes \End(\M)
\end{equation}
with respect to $\mathcal{M}$} (= {\em $F$ is normal}),
if for every object $X\in \mathcal{B}$ there exists a subobject $X_0\subseteq X$ such that $F(X_0)$ 
is the largest subobject of $F(X)$ contained in $\End(\M)\subseteq \C\boxtimes \End(\M)$.  
In this case we will also say that $\B$ is an {\it extension of $\C$ by $\A$ 
with respect to $\M$} \cite{EG}. 

Note that if $\M={\rm Vec}$, this definition coincides with that of \cite{BN}. 
In particular, if $H\subseteq G$ are finite groups, $\B:=\Rep(G)$, $\C:=\Rep(H)$,
$F$ is the restriction functor, $\A:={\rm Ker}(F)$, and 
$\M:={\rm Vec}$, then $F$ is normal if and only if $H$ is a normal subgroup of $G$, which motivates the terminology. 
Also, it is clear that if $\B=\C\boxtimes \A$ and $F$ is the obvious functor, 
then $F$ defines an exact sequence with respect to $\M$. 
So $\C\boxtimes \A$ is an extension of $\C$ by $\A$ with respect to any 
indecomposable $\A$-module category $\M$ (e.g., $\M=\A$). 

By \cite[Theorem 3.4]{EG}, (\ref{ES}) defines an exact sequence of fusion categories if and only if $\FPdim(\B)=\FPdim(\A)\FPdim(\C)$. 

Now let $\mathcal{N}$ be an indecomposable module category over $\mathcal{C}$. 
Then the multifusion category $\mathcal{C}\boxtimes \End(\M)$ acts on $\mathcal{N}\boxtimes \mathcal{M}$ componentwiselly, and we have natural isomorphisms
$$(\mathcal{C}\boxtimes \End(\M))_{\mathcal{N}\boxtimes \mathcal{M}}^*\cong\mathcal{C}_{\mathcal{N}}^*\boxtimes \End(\M)^*_{\M}\cong\mathcal{C}_{\mathcal{N}}^*\boxtimes \Vect\cong\mathcal{C}_{\mathcal{N}}^*.$$ By \cite[Theorem 4.1]{EG}, the dual sequence to (\ref{ES}) with respect to $\N\boxtimes \M$: 
\begin{equation}\label{ES1}
\mathcal{A}_{\mathcal{M}}^*\boxtimes \End(\mathcal{N})\xleftarrow{\iota^*} \mathcal{B}_{\mathcal{N}\boxtimes \mathcal{M}}^*\xleftarrow{F^*} \mathcal{C}_{\mathcal{N}}^*
\end{equation}
is exact with respect to $\N$.

\subsection{Group-theoretical fusion categories} Let $G$ be a finite group, and let $\omega\in Z^3(G,k^\times)$ be a $3$-cocycle. Let $\Vect(G,\omega)$ be the 
fusion category of finite dimensional $G$-graded vector spaces with associativity defined by $\omega$. The simple objects of $\Vect(G,\omega)$ are invertible and are in correspondence with elements $g$ of $G$.

Let $L\subseteq G$ be a subgroup, and let $\psi\in C^2(L,k^\times)$ be a $2$-cochain such that $d\psi=\omega|_{L}$. 
Let
$\mathcal{M}(L,\psi)$ be the corresponding indecomposable module category over $\Rep(G)$, and let
$\tilde{\mathcal{M}}(L,\psi)$ be the corresponding indecomposable module category over $\Vect(G,\omega)$ (see \cite[Subsection 8.8]{ENO1}).
Namely, $\mathcal{M}(L,\psi)$ is the category of representations of the twisted group algebra $k^{\psi}[L]$, while $\tilde{\mathcal{M}}(L,\psi)=\Mod_{\Vect(G,\omega)}(k^{\psi}[L])$ is the category of right $k^{\psi}[L]$-modules in $\Vect(G,\omega)$. The simple objects of $\tilde{\mathcal{M}}(L,\psi)$ are in correspondence with cosets $gL$ in $G/L$.

Recall \cite{O} that the group-theoretical fusion category $\C(G,\omega,L,\psi)$ is the dual category $\Vect(G,\omega)^{*}_{\tilde{\mathcal{M}}(L,\psi)}=\Bimod_{\Vect(G,\omega)}(k^{\psi}[L])$. For example, $\C(G,\omega,1,1)=\Vect(G,\omega)$ and $\C(G,1,G,1)=\Rep(G)$. By \cite[Theorem 3.1]{O}, the indecomposable module categories over $\C(G,\omega,L,\psi)$ are parametrized by the conjugacy classes of pairs $(L_1,\psi_1)$ where $L_1\subseteq G$ is a subgroup such that $\omega_{\mid L_1}=1$ and $\psi_1 \in H^2(L_1,k^{\times})$ is a cohomology class. Namely, the module category corresponding to a pair $(L_1,\psi_1)$ is the category $\M(L_1,\psi_1):=\Bimod_{\Vect(G,\omega)}(k^{\psi}[L],k^{\psi_1}[L_1])$. We have,
\begin{equation}\label{dualgt}
\C(G,\omega,L,\psi)^*_{\M(L_1,\psi_1)}=\C(G,\omega,L_1,\psi_1).
\end{equation}
For example, $\Rep(G)^{*}_{\mathcal{M}(L_1,\psi_1)}=\C(G,1,L_1,\psi_1)$.

\section{exact factorizations of fusion categories}

Let $\B$ be a fusion category, and let $\A,\C\subseteq \B$ be fusion subcategories of $\B$. Let $\A\C$ be 
the full abelian (not necessarily tensor) subcategory 
of $\B$ spanned by direct summands in $X\otimes Y$, where $X\in \A$ and $Y\in \C$. Observe that since $\A\C$ contains the unit object ${\bf 1}$ it follows that $\A\C$ is an indecomposable semisimple $\A\boxtimes \C^{op}$-submodule category of $\B$. 

\begin{definition}\label{facdef}
We say that $\B$ {\em factorizes} into a product of $\A$ and $\C$ if $\B=\A\C$. (Equivalently, if $\B$ is an indecomposable module category over $\A\boxtimes \C^{op}$.) 
\end{definition}

For a full abelian subcategory $\mathcal{E}\subseteq \A$, let 
$$\FPdim(\mathcal{E}):=\sum_{X\in \Irr(\mathcal{E})}\FPdim(X)^2,$$ and let $R_{\mathcal{E}}:=\sum_{X\in Irr(\mathcal{E})} \FPdim(X)X$. 
 
\begin{lemma} 
Let $\B$ be a fusion category, let $\A,\C\subseteq \B$ be fusion subcategories of $\B$, and 
let $\D:=\A\cap \C$. Then 
$$\FPdim(\A)\FPdim(\C)=\FPdim(\A\C)\FPdim(\D).$$
\end{lemma}

\begin{proof}
Since $\A\C$ is an indecomposable module category over $\A\boxtimes \C^{op}$, we have $R_{\mathcal{A}}R_{\mathcal{C}}=\lambda R_{\A\C}$, since 
both $R_{\A\C}$ and $R_{\mathcal{A}}R_{\mathcal{C}}$ are positive eigenvectors
for actions of $\A$ and $\C$, and such a vector is unique 
up to scaling (see Subsection 2.1). To find $\lambda$, let us use the inner product given by $(X,Y):=\dim \Hom(X,Y)$. We have 
$$(R_{\A\C},{\bf 1})=1,$$
and
$$
(R_{\mathcal{A}}R_{\mathcal{C}},{\bf 1})=(R_{\mathcal{A}},R_{\mathcal{C}})=
\sum_{X\in \Irr(\D)}\FPdim(X)^2=\FPdim(\D).$$ 
So $\lambda=\FPdim(\D)$, and 
$R_{\mathcal{A}}R_{\mathcal{C}}=\FPdim(\D)R_{\A\C}$.
Taking $\FPdim$ of both sides, we get the statement. 
\end{proof}

\begin{corollary}\label{compdim}
Let $\B$ be a fusion category, let $\A,\C\subseteq \B$ be fusion subcategories of $\B$, and 
let $\D:=\A\cap \C$. We have $$\FPdim(\B)\ge \frac{\FPdim(\A)\FPdim(\C)}{\FPdim(\D)},$$ 
and we have a factorization $\B=\A\C$ 
if and only if this inequality is an equality. \qed
\end{corollary}

\begin{definition}\label{exfacdef}
A factorization $\B=\A\C$ of a fusion category $\B$
into a product of two fusion subcategories $\A,\C\subseteq \B$ of $\B$ is called {\em exact} if $\A\cap \C=\Vect$, and is denoted by $\B=\A\bullet \C$.
\end{definition}

\begin{remark}\label{permute}
Note that since $(\A\C)^*=\C\A$, if $\B=\A\C$ then $\B=\C\A$ and if 
$\B=\A\bullet \C$ then $\B=\C\bullet \A$. 
\end{remark}

\begin{example}
Exact factorizations $\C=\C_1\bullet \C_2$, where $\C_i=\Vect(G_i,\omega_i)$ 
are arbitrary pointed fusion categories ($i=1,2$), 
are classified by groups $G$ with exact 
factorization $G=G_1G_2$ and $\omega\in H^3(G,k^{\times})$ 
which restricts to $\omega_1$ on $G_1$ and to $\omega_2$ on $G_2$. 
\end{example}

\begin{lemma}\label{simple}
Let $\B$ be a fusion category, and let $\A,\C\subseteq \B$ be fusion subcategories of $\B$ such that $\A\cap \C=\Vect$. Then for any simple objects $X\in \A$ and $Y\in \C$, $X\otimes Y$ is simple in $\B$ and it determines $X,Y$. 
\end{lemma}

\begin{proof}
For any simple objects $X'\in \A$ and $Y'\in \C$, we have $$\Hom(X\ot Y,X'\ot Y')=\Hom(X'^*\ot X,Y'\ot Y^*).$$
Note that $X'^*\ot X\in \A$ and $Y'\ot Y^*\in \C$. So if the unit object $\bf{1}$ is not contained in both $X'^*\ot X$ and $Y'\ot Y^*$, then $\Hom(X'^*\ot X,Y'\ot Y^*)=0$. Thus $\Hom(X'^*\ot X,Y'\ot Y^*)=0$ unless
$X=X'$ and $Y=Y'$. If so, then there is a single copy of $\bf{1}$ on each side, so $\Hom(X'^*\ot X,Y'\ot Y^*)$ is $1$-dimensional. Thus $X\ot Y$ is simple, and it determines $X,Y$, as claimed.
\end{proof}

\begin{theorem}\label{exfac}
Let $\B$ be a fusion category, and let $\A,\C\subseteq \B$ be fusion subcategories of $\B$. The following are equivalent:

(i) Every simple object of $\B$ can be uniquely expressed in the form $X\ot Y$, where $X\in \A$ and $Y\in \C$ are simple objects.

(ii) $\B=\A\bullet \C$ is an exact factorization in the sense of Definition \ref{exfacdef}. 
\end{theorem}

\begin{proof}
Suppose that (i) holds. Clearly, $\A\cap \C=\Vect$, so by Lemma \ref{simple}, for any simple objects $X\in \A$ and $Y\in \C$, $X\otimes Y$ is simple in $\B$. 
We therefore have,
$$\FPdim(\B)=\sum_{X\in \Irr(\A),\, Y\in \Irr(\C)} \FPdim(X\otimes Y)^2=\FPdim(\A)\FPdim(\C),$$
so by Corollary \ref{compdim}, (ii) holds. 

Conversely, suppose that (ii) holds. Then by Lemma \ref{simple}, for any simple objects $X\in \A$ and $Y\in \C$, $X\otimes Y$ is simple in $\B$ and it determines $X,Y$. Thus, $$\sum_{X\in \Irr(\A),\, Y\in \Irr(\C)} \FPdim(X\otimes Y)^2=\FPdim(\A)\FPdim(\C)=\FPdim(\B),$$
and we see that (i) holds.
\end{proof}

\begin{corollary}\label{bridedfus}
Suppose $\B$ is a braided fusion category with an exact factorization $\B=\A\bullet \C$. 
Then $\B=\A\boxtimes \C$ is a Deligne tensor product, and 
$\A$ and $\C$ projectively centralize each other in the sense of \cite[Section 8.22]{EGNO}.  
\end{corollary}

\begin{proof}
By Theorem \ref{exfac}, we have an equivalence of abelian categories $F:\A\boxtimes \C\xrightarrow{\cong} \B$, given by  $F(X\boxtimes Y)=X\otimes Y$ for every objects $X\in \A$ and $Y\in \C$. The equivalence functor $F$  has a tensor structure $J$ defined by the braiding structure $c$ on $\B$. Namely, for every $X,X'\in \A$ and $Y,Y'\in \C$, $$J_{X\boxtimes Y,X'\boxtimes Y'}:F((X\boxtimes Y)\otimes (X'\boxtimes Y'))\xrightarrow{\cong} F(X\boxtimes Y)\otimes F(X'\boxtimes Y')$$ is given by the isomorphism
$$
\id_X\otimes c_{X',Y}\otimes \id_{Y'}:(X\otimes X')\otimes (Y\otimes Y')\xrightarrow{\cong} (X\otimes Y)\otimes 
(X'\otimes Y').$$
It is straightforward to verify that since $c$ satisfies the braiding axioms, $J$ satisfies the tensor functor axioms.

Furthermore, it follows from Theorem \ref{exfac} that for every simple objects $X\in \A$ and $Y\in \C$, $$c_{Y,X}\circ c_{X,Y}:X\otimes Y\xrightarrow{\cong}X\otimes Y$$ is an automorphism of the {\em simple} object $X\otimes Y$ in $\B$. We therefore get that $c_{Y,X}\circ c_{X,Y}=\lambda\cdot \id_{X\otimes Y}$ for some $\lambda\in k^{\times}$, as claimed.
\end{proof}

We conclude this section with a couple of natural questions about the new notion
of exact factorization of fusion categories.

\begin{question}\label{questions}
(i) Let $\A,\C$ be fusion categories. What are the fusion categories $\B$ admitting 
an exact factorization $\B=\A\bullet \C$?

(ii) What can be said about the center $\Z(\B)$ of a fusion category $\B$ with an exact factorization $\B=\A\bullet \C$?
\end{question}

\begin{remark}
At the level of finite dimensional Hopf algebras, Question \ref{questions} (i) is called the ``factorization problem" (see, e.g., \cite{AM1,AM2,ABM} for a systematic study of the factorization problem for groups, Hopf algebras and Lie algebras). 
\end{remark}

\section{Extensions of fusion categories}

Retain the notation from Subsection 2.2.

We are now ready to state and prove our main result, in which we relate exact factorizations of fusion categories with exact sequences of fusion categories.

\begin{theorem}\label{mainnew}
The following hold:

(i) Let
$$\mathcal{A}\xrightarrow{\iota} \mathcal{B}\xrightarrow{F} \mathcal{C}\boxtimes \End(\N)$$
be an exact sequence of fusion categories with respect to the indecomposable $\A$-module category $\N$.
Then the dual fusion category $\mathcal{B}^*_{\C\boxtimes \mathcal{N}}$ admits an exact factorization $\mathcal{B}^*_{\C\boxtimes \mathcal{N}}=\C\bullet \A^*_{\N}$.

(ii) Any exact factorization $\B=\C\bullet \A$ of fusion categories defines an exact sequence of fusion categories with respect to any indecomposable $\A$-module category $\N$. In particular, it defines an exact sequence
$$\A\xrightarrow{\iota} \B\xrightarrow{F} \mathcal{C}\boxtimes \End(\A)$$
with respect to the indecomposable $\A$-module category $\A$. 
\end{theorem}

\begin{proof}
(i) Choose a nonzero object $N\in\mathcal{N}$, and consider the algebra object $A:=\underline{\End}_{\A}(N)$ in $\A$. Then $A$ is an algebra in $\mathcal{B}$.

Let $\M:=\Mod_{\mathcal{B}}(A)$ be the category of right $A-$modules in $\mathcal{B}$, and consider $\M$ as a left module category over $\mathcal{B}$ in the usual way. By \cite[Theorem 3.6]{EG}, $\M$ and $\C\boxtimes \mathcal{N}$ are equivalent as $\mathcal{B}-$module categories.

Since $\mathcal{B}^*_{\M}$ can be identified with the fusion category $\Bimod_{\mathcal{B}}(A)$, we see
that $\mathcal{B}^*_{\M}$ contains $\Bimod_{\A}(A)=\mathcal{A}^*_{\N}$ as a fusion subcategory.

Also, by taking the dual of $\mathcal{B}\xrightarrow{F} \C\boxtimes \End(\mathcal{N})$ with respect to the module category $\M=\C\boxtimes \mathcal{N}$, we get that $\mathcal{B}^*_{\M}$ contains $\C$ as a fusion subcategory.

Furthermore, by assumption, every object from $\mathcal{A}^*_{\N}$ is sent to $\End(\mathcal{N})$ under $F$, while every non-trivial object of $\C$ is not. This means that $\mathcal{A}^*_{\N}\cap \C=\Vect$ (inside $\mathcal{B}^*_{\M}$). Hence by Lemma \ref{simple}, the objects $X\otimes Y$ ($X\in \Irr(\mathcal{A}^*_{\N})$ and $Y\in \Irr(\C)$) are pairwise non-isomorphic simple objects in $\mathcal{B}^*_{\M}$. 

Finally, we have
\begin{eqnarray*}
\lefteqn{\sum_{X\in \Irr(\mathcal{A}^*_{\N}),\,Y\in \Irr(\C)}\FPdim_{\mathcal{B}^*_{\M}}(X\otimes Y)^2}\\
= & & \sum_{X\in \Irr(\mathcal{A}^*_{\N}),\,Y\in \Irr(\C)}\FPdim_{\mathcal{A}^*_{\N}}(X)^2 \FPdim_{\C}(Y)^2\\
= & & 
\FPdim(\mathcal{A}^*_{\N})\FPdim(\C)=
\FPdim(\A)\FPdim(\C).
\end{eqnarray*}
Since by \cite[Theorem 3.4]{EG}, $$\FPdim(\A)\FPdim(\C)=\FPdim(\mathcal{B})=\FPdim(\mathcal{B}^*_{\M}),$$ we conclude from Theorem \ref{exfac} that 
$\mathcal{B}^*_{\M}=\mathcal{A}^*_{\N}\bullet \C$ is an exact factorization, as desired.

(ii) Let $\N$ be an indecomposable $\A$-module category, and let $A$ be an algebra in $\A$ such that $\N=\Mod_{\A}(A)$. Since $A$ is an algebra in $\B$, we may consider the $\B-$module category $\M:=\Mod_{\B}(A)$ of right $A-$modules in $\B$. 
Observe that since $\M=\B\boxtimes_{\A}\N=(\C\bullet \A)\boxtimes_{\A}\N$, $\M=\C\boxtimes \N$ as a right $\C$-module category.

Now, let $\bar{\A}:=\A^*_{\N}$, $\bar{\B}:=\B^*_{\M}$ be the dual fusion categories 
of $\A$, $\B$ with respect to $\N$, $\M$, respectively. Then we have a sequence
$$\bar{\A}\xrightarrow{\iota} \bar{\B}\xrightarrow{F} \mathcal{C}\boxtimes \End(\N).$$
Since $\FPdim(\bar{\A})\FPdim(\C)=\FPdim(\bar{\B})$, and $\bar{\A}$ is in the kernel of $F$, it follows from \cite[Theorem 3.4]{EG} that this sequence of fusion categories is exact with respect to the indecomposable $\A$-module category $\N$, as desired. 

The proof of the theorem is complete.
\end{proof}

We can now deduce from Theorem \ref{mainnew} and (\ref{dualgt}) that any extension of a group-theoretical fusion category by another one is {\em Morita equivalent} to a fusion category with an exact factorization into a product of two group-theoretical fusion subcategories.

Retain the notation from Subsection 2.3.

\begin{corollary}\label{gtext}
Let $\C(G_1,\omega_1,L_1,\psi_1)$ and $\C(G_2,\omega_2,L_2,\psi_2)$ be two group-theoretical fusion categories, and let $\mathcal{M}=\mathcal{M}(L_3,\psi_3)$ be an indecomposable module category over $\C(G_1,\omega_1,L_1,\psi_1)$. Suppose 
$$\C(G_1,\omega_1,L_1,\psi_1)\xrightarrow{\iota}\mathcal{B}\xrightarrow{F} \C(G_2,\omega_2,L_2,\psi_2)\boxtimes \End(\mathcal{M})$$ is an exact sequence with respect to $\mathcal{M}$. Then the dual fusion category $\mathcal{B}^*_{\C(G_2,\omega_2,L_2,\psi_2)\boxtimes \mathcal{M}}$ admits an exact factorization $$\mathcal{B}^*_{\C(G_2,\omega_2,L_2,\psi_2)\boxtimes \mathcal{M}}=\C(G_1,\omega_1,L_3,\psi_3)\bullet \C(G_2,\omega_2,L_2,\psi_2)$$
into a product of two group-theoretical fusion subcategories. \qed
\end{corollary}

\begin{example}\label{main}
Let $G_1,G_2$ be finite groups. 

(i) Let $\mathcal{M}=\mathcal{M}(L_3,\psi_3)$ be an indecomposable module category over $\Rep(G_1)$, and suppose $$\Rep(G_1)\xrightarrow{\iota}\mathcal{B}\xrightarrow{F} \Vect(G_2,\omega_2)\boxtimes \End(\mathcal{M})$$ is an exact sequence with respect to $\mathcal{M}$. Then the dual fusion category $\mathcal{B}^*_{\Vect(G_2,\omega_2)\boxtimes \mathcal{M}}$ admits an exact factorization $$\mathcal{B}^*_{\Vect(G_2,\omega_2)\boxtimes \mathcal{M}}=\C(G_1,1,L_3,\psi_3)\bullet \Vect(G_2,\omega_2).$$

(ii) Let $\M=\M(L_3,\psi_3)$ be an indecomposable module category over $\Vect(G_1,\omega_1)$, and suppose  
$$
\Vect(G_1,\omega_1)\xrightarrow{\iota} \mathcal{B}\xrightarrow{F} \Vect(G_2,\omega_2)\boxtimes \End(\M)$$ 
is an exact sequence with respect to $\M$. 
Then the dual fusion category $\mathcal{B}^*_{\Vect(G_2,\omega_2)\boxtimes \M}$ admits an exact factorization $$\mathcal{B}^*_{\Vect(G_2,\omega_2)\boxtimes \M}=\C(G_1,\omega_1,L_3,\psi_3)\bullet \Vect(G_2,\omega_2).$$
\end{example}

\begin{corollary}\label{kacalg}
Let $G_1,G_2$ be finite groups.

(i) Suppose $$\Rep(G_1)\xrightarrow{\iota}\mathcal{B}\xrightarrow{F} \Vect(G_2,\omega_2)$$ is an exact sequence with respect to $\mathcal{M}(1,1)$ (= the standard fiber functor on $\Rep(G_1)$). Then 
$$\mathcal{B}=\Vect(G,\omega)^*_{\Vect(G_2,\omega_2)}=\C(G,\omega,G_1,1)$$ 
for some finite group $G$ with an exact factorization $G=G_1G_2$, and a $3$-cocycle $\omega\in H^3(G,k^{\times})$ that is trivial on $G_1$ and restricts to $\omega_2$ on $G_2$.

(ii) Suppose $$
\Vect(G_1,\omega_1)\xrightarrow{\iota} \mathcal{B}\xrightarrow{F} \Vect(G_2,\omega_2)\boxtimes \End(\tilde{\mathcal{M}}(1,1))$$ 
is an exact sequence with respect to $\tilde{\mathcal{M}}(1,1)$. Then 
$$\mathcal{B}=\Vect(G,\omega)^*_{\Vect(G_2,\omega_2)}=\C(G,\omega,G_1,1)$$ for some finite group $G$ with an exact factorization $G=G_1G_2$ and    
a $3$-cocycle $\omega\in H^3(G,k^{\times})$ that restricts to $\omega_1$ on $G_1$ and to $\omega_2$ on $G_2$.
\end{corollary}

\begin{proof}
(i) By specializing Example \ref{main} (i) to $\mathcal{M}=\M(1,1)$ we get that $$\mathcal{B}^*_{\Vect(G_2,\omega_2)}=\Vect(G_1)\bullet \Vect(G_2,\omega_2).$$ Since $\Vect(G_1)\bullet \Vect(G_2,\omega_2)$ is pointed, it is equal to $\Vect(G,\omega)$ for some finite group $G$ with an exact factorization $G=G_1G_2$, and a $3$-cocycle $\omega\in H^3(G,k^{\times})$ that is trivial on $G_1$ and restricts to $\omega_2$ on $G_2$. Thus, $$\B=\Vect(G,\omega)^*_{\Vect(G_2,\omega_2)}=\C(G,\omega,G_1,1),$$ as claimed. 

(ii) By specializing Example \ref{main} (ii) to $\M=\tilde{\M}(1,1)$ we get that $$\mathcal{B}^*_{\Vect(G_2,\omega_2)}=\Vect(G_1,\omega_1)\bullet \Vect(G_2,\omega_2).$$ Since $\Vect(G_1,\omega_1)\bullet \Vect(G_2,\omega_2)$ is pointed, it is equal to $\Vect(G,\omega)$ for some finite group $G$ with an exact factorization $G=G_1G_2$, and a $3$-cocycle $\omega\in H^3(G,k^{\times})$ that restricts to $\omega_1$ on $G_1$ and to $\omega_2$ on $G_2$. Thus, $$\B=\Vect(G,\omega)^*_{\Vect(G_2,\omega_2)}=\C(G,\omega,G_1,1),$$ as claimed.
\end{proof}

\begin{example}\label{kacalg2}
Let $B$ be a Kac (semisimple) Hopf algebra associated with an exact factorization $G=G_1G_2$ of finite groups (see e.g., \cite{K}, \cite{M}, \cite{Na}). Then $B$ fits into an exact sequence of semisimple Hopf algebras
$$\Fun(G_2)\to B\to k[G_1].$$ 
Let $\B:=\Rep(B)$. We have an exact sequence $$\Rep(G_1)\xrightarrow{\iota}\mathcal{B}\xrightarrow{F} \Vect(G_2)$$ with respect to the standard fiber functor on $\Rep(G_1)$.
By Corollary \ref{kacalg} (i), $\mathcal{B}=\C(G,\omega,G_1,1)$ for some $3$-cocycle $\omega\in H^3(G,k^{\times})$ that is trivial on $G_1$ and $G_2$ (but not necessarily on $G$). Thus every Kac Hopf algebra is group-theoretical. (This result was proved by Natale \cite{Na}.)

Conversely, Corollary \ref{kacalg} (i) says that if $\B$ fits into an exact sequence
$$\Rep(G_1)\xrightarrow{\iota}\mathcal{B}\xrightarrow{F} \Vect(G_2,\omega_2)$$ of fusion categories (with respect to the standard fiber functor on $\Rep(G_1)$), then $\mathcal{B}=\C(G,\omega,G_1,1)$ for some finite group $G$ with an exact factorization $G=G_1G_2$, and a $3$-cocycle $\omega\in H^3(G,k^{\times})$ that is trivial on $G_1$ and restricts to $\omega_2$ on $G_2$. Thus $\B$ is the representation category of a semisimple group-theoretical quasi-Hopf algebra (which may be referred to as the Kac quasi-Hopf algebra corresponding to the exact factorization $G=G_1G_2$ and the $3-$cocycle $\omega\in H^3(G,k^{\times})$, see e.g., \cite{S}).
\end{example}

\begin{example}\label{equiv} 
Let $\B=\C^{G}$ be a $G$-equivariantization of a fusion category $\C$. We have an exact sequence
$$\Rep(G)\xrightarrow{\iota}\mathcal{B}\xrightarrow{F} \C$$ with respect to the standard fiber functor on $\Rep(G)$ (see e.g., \cite[Example 4.4]{EG}). Thus, $\B^*_{\C}$ admits an exact factorization $\B^*_{\C}=\Vect(G)\bullet \C$. 
In \cite[Example 7.12.25]{EGNO} it is explained that 
$\Vect(G)\bullet \C$ in this case is the semidirect product discussed in that example. 
\end{example}

\begin{example}\label{Nikshych}
In Theorem \ref{main} one does not always get a group-theoretical fusion category. In particular, an exact factorization with two group-theoretical factors is not always group-theoretical
(even when one of the factors is pointed). 

For example, consider Nikshych's
example of the representation category $\B$ of a non group-theoretical semisimple Hopf algebra of dimension $4p^2$. By \cite[Corollary 4.6]{Ni}, $\B=\C^{\mathbb{Z}/2\mathbb{Z}}$ is a $\mathbb{Z}/2\mathbb{Z}$-equivariantization of a group-theoretical fusion category $\C$. So we have an exact sequence
$$\Rep(\mathbb{Z}/2\mathbb{Z})\xrightarrow{\iota}\mathcal{B}\xrightarrow{F} \C$$ with respect to the standard fiber functor on $\Rep(\mathbb{Z}/2\mathbb{Z})$. Hence the resulting exact factorization fusion category $\Vect(\mathbb{Z}/2\mathbb{Z})\bullet \C$ is not group-theoretical.
\end{example}

\begin{remark}
M. Mombelli and S. Natale obtained some related results for exact sequences of finite tensor categories in \cite[Section 7.2]{MN}.
\end{remark}

\begin{question}
If $\B=\A\bullet \C$ and $\A,\C$ are {\em weakly} group-theoretical fusion categories,
is it true that $\B$ is weakly group-theoretical? (See \cite{ENO2} for the definition of a weakly group-theoretical fusion category.) By Theorem \ref{mainnew}, this question 
is equivalent to the question if an extension of a weakly group-theoretical fusion category by another one is weakly group-theoretical. (Example \ref{Nikshych} shows that the answer is ``no" for group-theoretical fusion categories.)
\end{question}

\end{document}